\documentclass[12pt]{article}
\usepackage{amssymb,latexsym,amsmath,eufrak,amsthm,cite,pictex}
\usepackage[active]{srcltx}
\usepackage{hyperref}
\usepackage[inline]{enumitem}
\usepackage[T1]{fontenc}
\usepackage{mathrsfs} 
\usepackage{graphicx}
\usepackage{epstopdf}
\usepackage[inline]{enumitem}
\usepackage{tikz}
\usetikzlibrary{arrows, calc, patterns}
\usetikzlibrary{arrows.meta}
\usetikzlibrary{positioning}
\usetikzlibrary{bending}
\tikzset{>=triangle 45}
\allowdisplaybreaks[1]

\def\g{\gamma}

\def\m{\mu}
\def\s{\sigma}
\def\t{\tau}

\def\G{\Gamma}

\def\z{\zeta}

\def\Cs{Ces\'{a}ro}
\newcommand{\R}{{\mathbb R}}
\newcommand{\Z}{{\mathbb Z}}

\newcommand{\C}{{\mathbb C}}

\makeatletter
\newcommand{\inlineitem}[1][]{%
\ifnum\enit@type=\tw@
    {\descriptionlabel{#1}}
  \hspace{\labelsep}
\else
  \ifnum\enit@type=\z@
       \refstepcounter{\@listctr}\fi
    \quad\@itemlabel\hspace{\labelsep}
\fi}
\makeatother

\DeclareMathOperator{\im}{Im}
\theoremstyle{plain}

\newtheorem{lemma}{Lemma}
\newtheorem{proposition}{Proposition}
\theoremstyle{definition}

\theoremstyle{remark}

\def\Cs{Ces\'{a}ro}
\title{Estimates of $\G$-functions and their $\log{}$-derivatives}
\author{Miron Bekker and Boris Mityagin}
\date{}
\begin{document}
\maketitle
\begin{abstract}
We provide a lower bound for a weighted ratio of $\G-$functions and its $\log{}-$derivative. Then we apply this result to estimate  the norm of the operator associated with the {\Cs} operator in the Hardy space $H^p(\C_+)$, $1<p<\infty$.
\end{abstract}
Hermite-Hadamard Inequality, Hardy Spaces, Cesaro operator.
\vskip 1.0truecm

The goal of this note is to prove a couple of inequalities related to $\G$-functions and their ($\log{}$) derivatives and to apply them to analysis of the norm
of the adjusted {\Cs} operator in Hardy spaces.

{\bf 1.} Let us define a smooth function
\begin{equation}\label{Phi-function}
  \Phi(x)=\begin{cases}
            2(1-\dfrac{1}{x})^{x-1}\dfrac{\G(1/2+x/2)\G(x)}{\G(1+x/2)\G(-1/2+x)}&x>1;\\
            \dfrac{4}{\pi}&x=1,
          \end{cases}
        \end{equation}
        on $[1,\infty)$.
        
        \begin{proposition}
          For $s\ge -1$, $s\ne 0$, the strict inequality
          \begin{equation}\label{Phi-R_inequality}
\log{\Phi(s+2)}>R(s),
\end{equation}
where
          \begin{equation}\label{R-function}
            R(s)=\frac{1}{24}\frac{s^2}{\kappa(\kappa+s)(\kappa+2s)}, \text{and}\quad \kappa=1+1/\sqrt{12},
          \end{equation}
 holds. Moreover, $\Phi(2)=1$.
\end{proposition}
{\bf Proof.} Applying the Legendre duplication formula to $\Phi(x)$ and taking the logarithmnic derivative of both sides (see \cite{AbSt}, formulas 6.1.18, 6.3.1, and 6.3.16; also Appendix, {\bf B}), we arrive to the function
\begin{gather}\label{log-dev1}
  F(s)\equiv\frac{d}{ds}\left(\log{\Phi(s+2)}\right)=\frac{\Phi^{\prime}(s+2)}{\Phi(s+2)}=\\
  \log{\left(\frac{1+s}{1+s/2}\right)}-\sum\limits_{n=0}^{\infty}g(n;s), \label{4}
\end{gather}
where
\begin{equation}\label{g-function}
  g(t;s)=\frac{1}{t+3/2+s/2}-\frac{1}{t+3/2+s},\quad -1<s<+\infty,\quad -\frac{1}{2}\le t<+\infty.
\end{equation}
Notice that the second derivative 
\begin{equation}\label{second_der}
\frac{d^2}{dt^2}g(t;s)=2\left[\frac{1}{(t+3/2+s/2)^3}-\frac{1}{(t+3/2+s)^3}\right]
\end{equation}
is well defined and is\\
(a) strictly positive on $[-3/2,+\infty)$ if $s>0$,
and\\
(b) strictly negative for $t\in[-1/2,+\infty) $ if $-1<s<0.$\\
Therefore 
\begin{subequations}\label{convexity}
\begin{align}
g(\cdot;s)\quad \text{is strictly convex on} \quad [-3/2, +\infty)\quad \text{if} \quad s>0; \label{convex}\\
g(\cdot; s) \quad \text{is strictly concave on}\quad [-1/2,+\infty) \quad \text{if}\quad -1<s<0. \label{concave}    
\end{align}
\end{subequations}
\vskip 1.0truecm
If $h(t)$ is a strictly convex (concave) function on the interval $[a,b]$ then 
\begin{equation}
h(\frac{a+b}{2})<\frac{1}{b-a}\int\limits_a^b h(t)dt,\tag{9a}\label{H-H_convex}\\
\end{equation}
respectively
\begin{equation}
h(\frac{a+b}{2})>\frac{1}{b-a}\int\limits_a^b h(t)dt,\tag{9b}\label{H-H_concave}.
\end{equation}
This is the Hermite-Hadamard inequality (see  \cite{HLP34}, Chapter 3, Statement 125, or \cite{NP18}, Sec.1.10, formula (1.32)).
Therefore, if $s>0$, for any $n\in\Z_+$ 
\begin{equation}
g(n;s)<\int\limits_{n-1/2}^{n+1/2}g(t;s)dt \tag{10a}\label{9a}
\end{equation}
and
\begin{equation*}
\sum\limits_{n=0}^{\infty}g(n;s)<\sum\limits_{n=0}^{\infty}\int\limits_{n-1/2}^{n+1/2}g(t;s)dt=\int\limits_{-1/2}^{\infty}g(t;s)dt=
\log{\left(\frac{1+s}{1+s/2}\right)},
\end{equation*}
and $F(s)>0$ if $s>0$.

If $-1<s<0$ we  apply the Hermite-Hadamard inequality \eqref{H-H_concave} and conclude by repeating an analogue of \eqref{9a} that 
\begin{equation}
F(s)=\log{\left(\frac{1+s}{1+s/2}\right)}-\sum\limits_{n=0}^{\infty}g(n;s)<0.\tag{10b}\label{9b}
\end{equation}
With $F(0)=0$ or $\Phi^{\prime}(2)=0$, and $\Phi(2)=1$, we can state as an intermidiate statement the following\\
{\bf Claim 2.} {\it Under the notations \eqref{Phi-function} 
\begin{equation}\label{claim2}
\log{\Phi(x)}\begin{cases}
>0&\text{if}\quad +1\le x<\infty, \quad x\ne 2,\\ \tag{11}
=0&\text{if}\quad x=2.
\end{cases}
\end{equation} 
}
\vskip 1.0truecm
{\bf 3.} But we want to go farther and make inequality \eqref{claim2} quantitative, i.e. put as the right hand side some elementary functions of $p$ or $s=p-2$.

Now we use Peano kernel formula of the midpoint integration rule. The gap
\begin{equation}\label{gup}
E(f):=\frac{1}{b-a}\int\limits_a^bf(x)dx-f\left(\frac{a+b}{2}\right)\tag{12}
\end{equation}
is a linear functional on $C^2([a,b])$, and it has a representation 
\begin{equation*}
E(f)=\int\limits_a^bK(t)f^{\prime\prime}(t)dt,
\end{equation*}
where
\begin{equation}\label{Peano_Kernel}
K(t)=
\begin{cases}
\dfrac{1}{2(b-a)}(t-a)^2,&a\le t\le \m:=\dfrac{a+b}{2},\\ \tag{13}
\dfrac{1}{2(b-a)}(t-b)^2,&\m\le t\le b
\end{cases}
\end{equation}
(see \cite{GP13}; \cite{RS11}, Sec. 13.2.2, formulas (13.49)-(13.51)).

Peano kernel on the interval $[-1/2,1/2]$ is 
\begin{equation*}
k(t)=\frac{1}{2}\left(\frac{1}{2}-|t|\right)^2,
\end{equation*}
it is strctly positive. 
Let us extend this function as a periodic one, with period $1$, i.e. 
\begin{equation*}
k(t+m)=k(t),\quad |t|\le 1/2, \quad m\in\Z_+, 
\end{equation*}
then $F(s)$ as a total gap \eqref{4} can be represented in the following way: 
\begin{equation*}
F(s)=\int\limits_{-1/2}^{\infty}k(t)g^{\prime\prime}(t;s)dt.
\end{equation*}
{\it Case $s>0$}. We have 
\begin{equation*}
(-1)^m\frac{d^m}{dt^m}g(t;s)>0,\quad m=0,1,2,3,
\end{equation*}
so $g^{\prime\prime}(\cdot;s)$ is monotone, strictly decreasing.  

Choose $\t$, $0<\t<1/2$, in such a way that areas of sets
\begin{equation*}
\left\{0\le t\le\t, k(\t)\le y \le k(t)\right\}
\end{equation*}
and 
\begin{equation*}
\left\{\t\le t\le 1/2, k(t)\le y \le k(\t)\right\}
\end{equation*}
are equal, i.e., with $k$ being even and $1-$periodic,
\begin{equation}\label{equal_areas}
\int\limits_{-\t}^{\t}[k(t)-k(\t)]dt=\int\limits_{\t}^{1-\t}[k(\t)-k(t)]dt\tag{14}
\end{equation}
Then
\begin{equation}\label{equal_area_2}
\frac{1}{2}\int\limits_0^{1/2}\left(\frac{1}{2}-t\right)^2dt=\frac{1}{2}\left(\frac{1}{2}-\t\right)^2\frac{1}{2} \tag{15}
\end{equation}
and 
\begin{equation}\label{tau-value}
\t=\frac{1}{2}\left(1-\frac{1}{\sqrt{3}}\right), \qquad k(\t)=\frac{1}{2}\left(\frac{1}{2\sqrt{3}}\right)^2=\frac{1}{24}. \tag{16}
\end{equation}
\begin{figure}[!h]
 \begin{center}
  \includegraphics[width=5in]{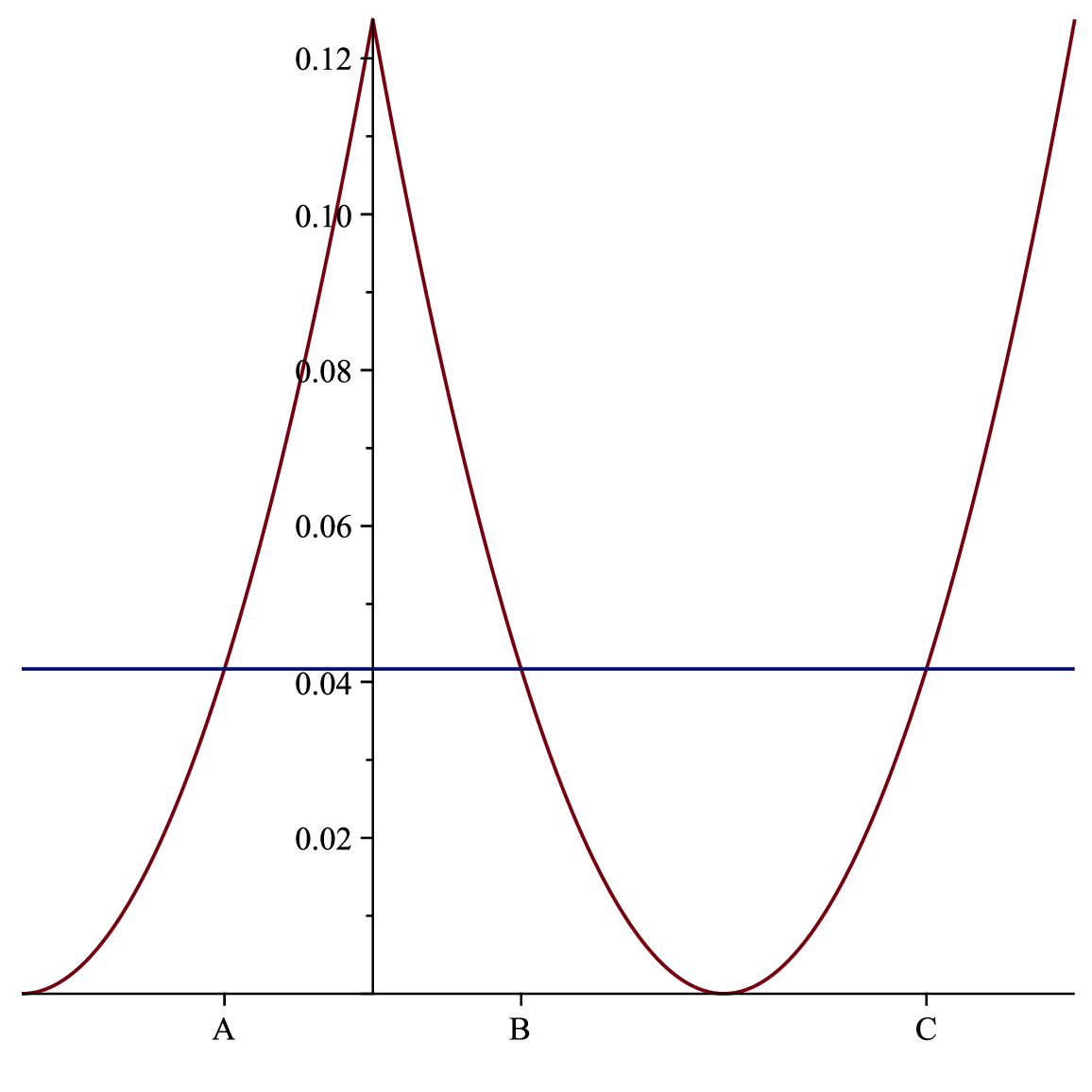}
 \end{center}
\caption{\small Extended Peano kernel on $[-1/2,1]$. Horizontal line represents $k(\t)=1/24$. Points $A,B,C$ correspond to values $-\t$, $\t,$ and $1-\t$, respectively, $\t=1/2(1-1/\sqrt{3})$.}
\end{figure}

Let us notice that if $(X,\mu)$, $(Y,\nu)$ are spaces with positive measures such that 
\begin{equation}\label{star}
\int_Xd\mu=\int_Yd\nu, \tag{17}
\end{equation}
 and $\exists c$ , $p(x)>c>q(y)$ $\forall x\in X$, $\forall y\in Y$, then
\begin{equation}\label{6.2}
\int\limits_Xp(x)d\mu>c\int\limits_Xd\mu=c\int\limits_Yd\nu>\int\limits_Yq(y)d\nu. \tag{18}
\end{equation}
This remark, with $X=(-\t,\t)$, $\mu=[h(x)-h(\t)]dx$,  $Y=(\t, 1-\t)$, $d\nu=[h(\t)-h(y)]dy$, $c=g^{\prime\prime}(\t)$, and \eqref{equal_areas} being \eqref{star}, where $p$ and $q$ coincide with $g^{\prime\prime}(t)$ on the corresponding intervals, leads to
\begin{equation*}
\int\limits_{-\t}^{\t}[k(t)-k(\t)]g^{\prime\prime}(t;s)dt>\int\limits_{\t}^{1-\t}[k(\t)-k(t)]g^{\prime\prime}(t;s)dt
\end{equation*}
or
\begin{equation*}
\int\limits_{-\t}^{1-\t}k(t)g^{\prime\prime}(t;s)dt>k(\t)\int\limits_{-\t}^{1-\t}g^{\prime\prime}(t;s)dt
\end{equation*}
The same holds for any shifted interval $[-\t, 1-\t]+m$, $m\in\Z_+$, so 
\begin{gather}\label{12.1}
F(s)=\int\limits_{-1/2}^{\infty}k(t)g^{\prime\prime}(t;s)dt>\tag{19} \\ 
\int\limits_{-\t}^{\infty}k(t)g^{\prime\prime}(t;s)dt>\frac{1}{24}\int\limits_{-\t}^{\infty}g^{\prime\prime}(t;s)dt=-\frac{1}{24}g^{\prime}(-\t).\nonumber
\end{gather}
Put $\kappa=3/2-\t=1+1/\sqrt{12}$. The straightforward calculation shows that 
\begin{equation*}
-g^{\prime}(-\t)=\frac{1}{(\kappa+s/2)^2}-\frac{1}{(\kappa+s)^2}
\end{equation*}
and
\begin{gather*}
\log{\Phi(s+2)}=\int\limits_0^sF(\s)d\s>\\
\frac{1}{24}\int\limits_0^s\left[\frac{1}{(\kappa+\s/2)^2}-\frac{1}{(\kappa+\s)^2}\right]d\s=\frac{1}{24}\frac{s^2}{\kappa(\kappa+s)(2\kappa+s)}.
\end{gather*}
\vskip 1.0truecm
{\it Case $s<0$}. The function $\g(t)=-g(t;s)$ is strictly convex on $[-1/2,\infty)$ by \eqref{concave}, and by \eqref{12.1}
\begin{equation}\label{12.2}
-F(s)=\int\limits_{-1/2}^{\infty}k(t)\left[-\frac{d^2}{dt^2}g(t;s)\right]dt>\frac{1}{24}g^{\prime}(-\t;s). \tag{20}
\end{equation}
Therefore,
\begin{gather*}
\log{\Phi(s+2)}=\int\limits_s^0[-F(\s)]d\s>\\
\frac{1}{24}\int\limits_s^0\left[\frac{1}{(\kappa+\s)^2}-\frac{1}{(\kappa+\s/2)^2}\right]d\s=\frac{1}{24}\frac{s^2}{\kappa(\kappa+s)(2\kappa+s)}.
\end{gather*}
\hfill Q.E.D.
\vskip 1.0truecm
It is interesting to compare functions $\log{\Phi(s+2)}$ and $R(s)$ given by \eqref{Phi-function}-\eqref{R-function} (see Figure 2). At infinity, by Stirling asympotic formula 
\begin{equation*}
\lim\limits_{s\to\infty}\log{\Phi(s+2)}=\frac{3}{2}\log{2}-1=0.039720771\ldots
\end{equation*}
and 
\begin{equation*}
\lim\limits_{s\to\infty}R(s)=\frac{1}{24\kappa}=\frac{1}{12(2+1/\sqrt{3})}=0.032332948\ldots
\end{equation*}
\begin{figure}[!h]
 \begin{center}
  \includegraphics[width=5in]{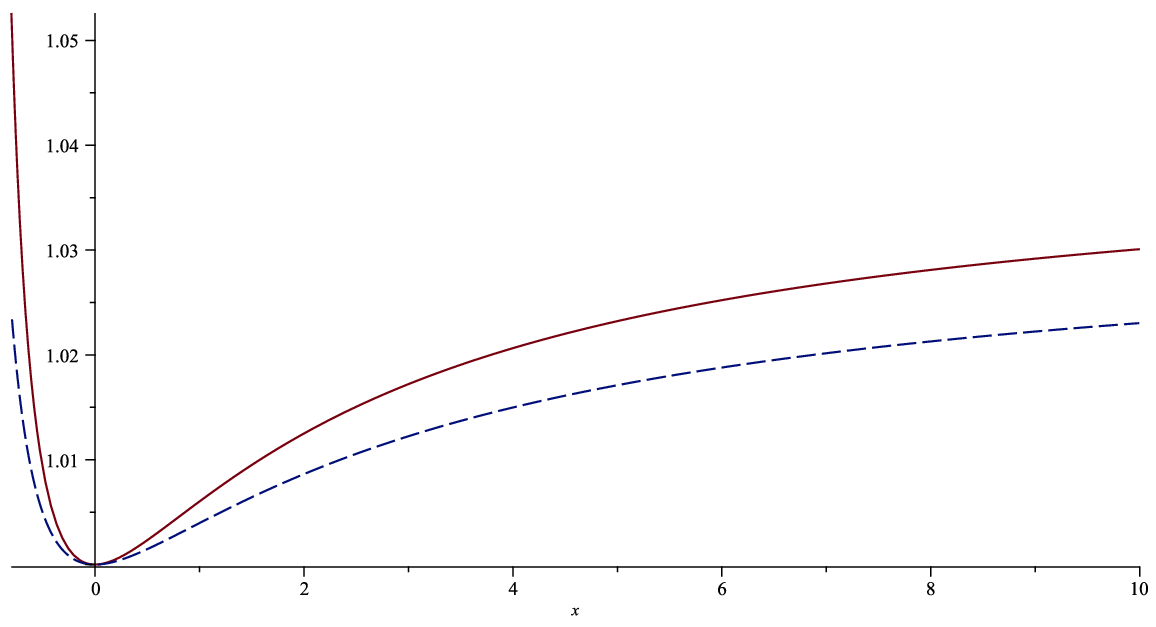}
 \end{center}
\caption{\small $\Phi(s+2)$( solid line) and $\exp{(R(s))}$ (dash line) for $-0.8\le s\le 10$}
\end{figure}
\vskip 1.0truecm
{\bf 4. Applications.} The function $\Phi(x)$,  defined by the formula \eqref{Phi-function}, appeared in \cite{ABC}, where the authors made an attempt to show that the norm  of the adjusted {\Cs} operator 
\begin{equation}\label{13}
V(p)=2\frac{p-1}{p}C-I,\quad 1<p<\infty,\quad p\ne 2 \tag{21}
\end{equation}
in Hardy space $H^p(\C_+)$, exceeds its spectral radius which is equal $1$. (To make our note self-contained we provide a few facts about {\Cs} operator 
\begin{equation}\label{14}
C:f\to \frac{1}{z}\int\limits_0^zf(\z)d\z \tag{22}
\end{equation}
in Appendix). As a step to this direction, in \cite{ABC} the test function 
\begin{equation}\label{15}
h(z)=\frac{1}{(i+z)^2}. \tag{23}
\end{equation}
was used.
The staightforward calculations (we leave them to a reader or refer to formulas (5)-(7)  in \cite{ABC}) lead to the following statement:
\begin{lemma} On the Hardy space $H^p(C_+)$, $1<p<\infty$, we have
\begin{equation*}
\Vert V(p)\Vert^p_p\ge \frac{\Vert V(p)h\Vert_p^p}{\Vert h\Vert^p_p}\ge \Phi(p) 
\end{equation*}
with equality only for $p=2$, where $\Phi$ is defined by \eqref{Phi-function}.
\end{lemma}
Now, with our results of Sections 1 and 2 above, we can claim the following.
\begin{proposition}
For the operator $V(p)$ defined by \eqref{13} in  Hardy space \\$H^p(\C_+)$, $1<p<\infty$, $p\ne 2$, the norm $\Vert V(p)\Vert_p>1$. Moreover 
\begin{equation*}
\Vert V(p)\Vert_p^p>\exp{(R(p-2))}
\end{equation*}
where 
\begin{equation*}
R(s)=\frac{1}{24}\frac{s^2}{\kappa(\kappa+s)(2\kappa+s)},\quad \kappa=1+1/\sqrt{12}.
\end{equation*}
\end{proposition} 
It was proved in \cite{ABC1} that for $p=2$ the operator $V(2)$ is unitary.

{\bf 5. Appendix.} In this Appendix we provide a brief information about Hardy spaces $H^p(\C_+)$ in the upper half-plabe $\C_+=\{z:\im z>0\}$ and the {\Cs} operator in these spaces. For detailed information about Hardy spaces we refer a reader to \cite{Dur1} (Chapter 11),  \cite{Koo} (in particular, Chapter VI, Sect. C, p. 112 ). We also recall a couple of  identities related to $\G-$function and its $\log{}-$deraivative that were used above.

{\bf A.} The Hardy space $H^p(\C_+)$, $1\le p<\infty$ consists of functions $f(z)$, which are analytic in the upper half-plane $\C_+$ and staitsfy condition 
\begin{equation}\label{norm}
\left[\sup\left\{\frac{1}{2\pi}\int\limits_{-\infty}^{\infty}\vert f(x+iy)\vert^pdx:y>0\right\}\right]^{1/p}<\infty.\tag{24}
\end{equation}
The space $H^{\infty}(\C_+)$ consists of functions which are analytic and bounded in $\C_+$.  $H^{p}(\C_+)$ are Banach spaces. For $1\le p<\infty$ the norm is defined  by the left side of \eqref{norm},
and $\Vert f\Vert_{\infty}=\sup\left\{\vert f\vert: z\in\C_+\right\}$. For a function $f\in H^p(\C_+)$ the {\it nontangential} limit $\lim\limits_{y\to 0}f(x+iy)=f^*(x)$ exists for almost all $x$, and the function $f^*\in L^p(\R)$.  The function $f(z)$ can be restored from $f^*(x)$ with the help of Cauchy formula.

The {\Cs} operator on the Hardy space $H^p(\C_+)$ in the upper half-plane was introduced in \cite{AS}. It is defined by the expression \eqref{14}.
In \cite{AS} it was proved that for $1<p\le\infty$ the operator $C$ is bounded on $H^p(\C_+)$ (another proof was given in \cite{ABC1}) and for $1<p<\infty$ its spectrum $\s(C)$ is given by
\begin{equation}
 \s(C)=\{w:|w-\frac{p}{2(p-1)}|=\frac{p}{2(p-1)}\}.\tag{25}
\end{equation}
From the last formula it follows that the spectrum of the operator $V(p)$, defined by \eqref{13}, is the unit circle.

{\bf B.} 
1. Legendre duplication formula (\cite{AbSt}, formula 6.1.18)
\begin{equation*}
\G(2z)=(2\pi)^{-1/2}2^{2z-1/2}\G(z)\G(z+\frac{1}{2});
\end{equation*}
2.  A series representation of the logarithmic derivative of $\G-$function (\cite{AbSt}, formula 6.3.16)
\begin{equation*}
\frac{d}{dx}\log{\G(x+1)}=\psi(1+x)=-\g+\sum\limits_{n=1}^{\infty}\frac{x}{n(n+x)},\quad n\ne -1,-2 \ldots,
\end{equation*}
where $\g$ is the Euler's constant.

{\bf 6. Acknowledgement}. The authors are very thankful to Professor Marina Chugunova, Claremont Graduate University, for a discussion of a few topics of numerical analysis.
 
 Miron Bekker, Department of Mathematics, the University of \\Pittsburgh at Johnstown,
Johnstown, PA 15904, USA, {\it bekker@pitt.edu}\\

Boris Mityagin, Department of Mathematics, The Ohio State University, 231 West 18th
Ave, Columbus, OH 43210, USA\\
{\it mityagin.1@osu.edu, boris.mityagin@gmail.com}
\end{document}